\documentclass[12pt]{amsart}

\usepackage{fullpage}

\newtheorem{thm}{Theorem}[section]

\newtheorem{lem}[thm]{Lemma}

\newtheorem{conj}[thm]{Conjecture}
\theoremstyle{definition}

\theoremstyle{remark}
\newtheorem{rem}[thm]{Remark}
\numberwithin{equation}{section}

\newcommand{\F}{\mathbb{F}}

\newcommand{\Z}{\mathbb{Z}}

\begin{document}

\title{Expansions in completions of global function fields}
\author{Chunlin Wang}
\address{School of Mathematical Sciences, Sichuan Normal University, Chengdu 610066, China}
\email{c-l.wang@outlook.com}

\thanks{}%

\begin{abstract}
It is well known that any power series over a finite field represents a rational function if and only if its sequence of coefficients is ultimately periodic. The famous Christol's Theorem states that a power series over a finite field is algebraic if and only if its sequence of coefficients is $p$-automatic. In this paper, we extend  these two results to expansions of elements in the completion of a global function field under a nontrivial valuation. As application of our generalization of Christol's theorem, we answer some questions about $\beta$-expansions of formal Laurent series over finite fields.  
\end{abstract}
\maketitle
\section{Introduction}

Let $p$ be a prime number and $q$ be a power of $p$. Denote by $\F_q$ the finite field of $q$ elements. The polynomial ring, power series ring and rational function field over $\F_q$ are denoted by $\F_q[z]$, $\F_q[[z]]$ and $\F_q(z)$ respectively. Let $\F_q((z))$ be the quotient field of $\F_q[[z]]$. For $$f(z):=\sum_{n=m}^{+\infty}a_nz^n\in\F_q((z)),$$ 
we say that $f$ is {\it rational} (over $\F_q(z)$) if $f(z)=P(z)/Q(z)$ for some polynomials $P,Q$ over $\F_q$, and say $f$ is {\it algebraic} (over $\F_q(z)$) if there exists a polynomial $R(z,w)\in\F_q[z,w]$ such that $R(z,f(z))=0$. 

Let $S:=\{a_n\}_{n\ge m}$ be the sequence of coefficients of $f$. The sequence $S$ is called {\it ultimately periodic} if there exist integers $N\ge m$ and $L>0$ such that $a_n=a_{n+L}$ for all $n\ge N$. For integers $i,j,k$ with $k>0, i\ge 0$, and $0\le j<k^i$, denote by $S_{i,j}$ the subsequence $\{a_{nk^i+j-m}\}_{n\ge 0}$ of $S$. We call $S$ {\it $k$-automatic} if its {\it $k$-kernel}, i.e. the set of subsequences  
$$\{S_{i,j}: i\ge 0, 0\le j<k^i\},$$ 
is finite. It is worth mentioning that this definition is just one of the equivalent definitions of automatic sequences. For more on automatic sequences, we refer the reader to Allouche and Shallit \cite{AS}, and references there. \\

For power series over $\F_q$, we have the following two well known facts about $f$ and its sequence of coefficients. 

{\it Fact A}: $f$ is rational if and only if $S$ is ultimately periodic. 

{\it Fact B}: $f$ is algebraic if and only if $S$ is $p$-automatic. \\
Fact A is elementary but useful. Fact B, known as Christol's Theorem (see \cite{Ch},\cite{CKMFR}), builds a deep connection between computer science and arithmetic in global function fields. 

Generalizations and analogues of Christol's theorem are considered in literature. Salon \cite{Sa87,Sa89} generalized Christol's Theorem to multivariable power series over a finite field. Independently, Sharif and Woodcock \cite{SW}, and Harase \cite{Ha88} generalized Christol's Theorem to the case of multivariable power series over an infinite field of positive characteristic. An extension of Christol's theorem to  generalized algebraic power series is given by Kedlaya \cite{Ke}. Quantitative version of Christol's theorem is studied in  \cite{Ha88, Ha89, FKdM, AB, Br} and \cite{AY}. The automaticity of $\beta$-expansion of formal Laurent series is considered in \cite{HM} and \cite{SS}.   
\\

Seen form the viewpoint of valuation theory, the power series ring $\F_q[[z]]$ is just the completion of $\F_q[z]$ under the valuation associated to the prime $z\F_q[z]$, and the power series expansion of a rational function $P/Q$, where $P,Q\in \F_q[z]$, is just a special case of expansions of elements in the completion of a global function field under a nontrivial valuation. This inspires us to consider generalizations of Fact A and B for expansions of elements in the completion of a global function field. 

Let $K$ be a {\it global function field} of characteristic $p$, i.e., $K$ is a field containing $\F_p$ and an element $x$ which is transcendental over $\F_p$ with $[K:\F_p(x)]<\infty$. The algebraic closure of $\F_p$ in $K$ is called the constant field of $K$. A valuation on $K$ is a function $v:K^*\to \Z$ satisfying   
    
    (1). $v(a+b)\ge \inf\{v(a),v(b)\}$, 
    
    (2). $v(ab)=v(a)+v(b)$. \\ 
We say $v$ is nontrivial if $v(x)\neq 0$ for at least one $x\in K^*$, and say $v$ is normalized if $v(K^*)=\mathbb{Z}$. For convenience we define $v(0):=+\infty$. Let $A:=\{x\in K: v(x)\ge 0\}$ be the valuation ring of $v$ and $M:=\{x\in K: v(x)>0\}$ be the maximal ideal of $A$. The valuation $v$ induces a topological structure on $K$, under which $K$ has a completion, denoted by $K_v$. Respectively denote by $A_v$ and $M_v$ the completions of $A$ and $M$. Hereafter we always suppose $v$ is normalized unless otherwise stated. For any $\pi\in K_v$ with $v(\pi)>0$, denote by $r:=|A_v/\pi A_v|$. Let 
$$\Gamma:=\{\gamma_1,\dots,\gamma_r\}\subset A_v$$ 
be a {\it complete set of representatives} of $A_v/\pi A_v$, that is, $\{\bar{\gamma_i}: 1\le i\le n\}=A_v/\pi A_v$, where $\bar{\gamma_i}$ denotes the reduction of $\gamma_i$ modulo $\pi A_v$. Then every element $x\in K_v$ can be uniquely expressed as an expansion 
\begin{equation} \label{expansion}
x=\sum_{n\ge m}a_n\pi^n, 
\end{equation} 
where $m$ is an integer depending only on $x$, and $a_n\in \Gamma$ for all $n\ge m$. We call (\ref{expansion}) the expansion of $x$ with respect to $(\Gamma,\pi)$, or simply expansion of $x$ if $(\Gamma,\pi)$ is clear from the context. The sequence of coefficients of (\ref{expansion}) is $\{a_n\}_{n\ge m}$. An element $x\in K_v$ is called rational if $x\in K$, and is called algebraic if $x$ is algebraic over $K$. 
We say $(\Gamma,\pi)$ has {\it property A} (respectively, {\it property B}) if Fact A (respectively, Fact B) holds for expansions with respect to $(\Gamma,\pi)$. 
Our main purpose is to find equivalent description for $(\Gamma,\pi)$ that has property A or B. This goal is achieved for $(\Gamma,\pi)$ having property A. 

\begin{thm}\label{Property A}
$(\Gamma,\pi)$ has property A if and only if the following two are true:

{\rm (i)} $\pi$ and elements of $\Gamma$ are rational.  

{\rm (ii)} $[K:\F_q(\pi)]=[K_v:\F_q((\pi))]$. 

\end{thm}

For $(\Gamma,\pi)$ with property B, we have a conjectural result. 

\begin{conj}\label{Property B}
$(\Gamma,\pi)$ has property $B$ if and only if $\pi$ and elements of $\Gamma$ are algebraic over $K$.  
\end{conj} 

It is easy to show $\pi$ and elements of $\Gamma$ are algebraic if $(\Gamma,\pi)$ has property B (see Lemma \ref{necessity of B}). An equivalent but more concrete statement of Conjecture \ref{Property B} is given at the end of Section 3. Though we could not verify Conjecture \ref{Property B} for its full generality, we do confirm it when $\Gamma$ is additively closed (see Theorem \ref{add-cls case of B}), where we say $\Gamma$ is additively closed if $\mu+\nu\in \Gamma$ for all $\mu,\nu\in\Gamma$.  

The paper is organized as follows. In section 2, we study expansions of rational elements and prove Theorem \ref{Property A}. In section 3, we study expansions of algebraic elements and prove Conjecture \ref{Property B} for additively closed $\Gamma$, which gives a generalization of Christol's theorem. In section 4, we use our generalization of Christol's theorem to prove a result on $\beta$-expansions of formal Laurent series (see Theorem \ref{beta expansions}), which answers questions posed in \cite{HM} and \cite{SS}. 

\subsection*{Notations} Let $K$ be a global function field with constant field $\F_q$, and $v$ be a normalized nontrivial valuation on $K$. The valuation ring of $K$ associated to $v$ is denoted by $A$, and the maximal ideal of $A$ is denoted by $M$. Let $K_v, A_v, M_v$ be completions of $K, A, M$ respectively. Write $f:=[A_v/M_v:\F_q]=[A/M:\F_q]$. Let $\pi$ be an element of $K_v$ with $v(\pi)=e>0$. Let $\Gamma\subset A_v$ be a complete set of representatives of $A_v/\pi A_v$. 

\section{Expansions of rational elements} 
Denote by $\F_q(\pi)$ the field generated by $\pi$ over $\F_q$, and by $\F_q((\pi))$ the completion of $\F_q(\pi)$ in $K_v$. Then  $[K_v:\F_q((\pi))]=ef$ (cf. \cite[Theorem 4.14]{B}). Moreover, $\pi A_v=M_v^e$, which implies $|A_v/\pi A_v|=|A_v/M_v|^e=q^{ef}$. Thus $A_v/\pi A_v$ is an $\F_q$-linear space of dimension $ef$. Furthermore if $\pi\in A$, then $\F_q(\pi)$ is a subfield of $K$ and $A/\pi A\cong A_v/\pi A_v$. Choose $\alpha_1,\dots,\alpha_{ef}\in A$ whose reductions modulo $\pi$ are $\F_q$-linearly independent, i.e., the reductions of $\alpha_1,\dots,\alpha_{ef}$ modulo $\pi$ consist of an $\F_q$-basis of $A/\pi A$. Then $\alpha_1,\dots,\alpha_{ef}$ are $\F_q((\pi))$-linearly independent, hence consist of an $\F_q((\pi))$-basis of $K_v$. On the other hand, $\alpha_1,\dots,\alpha_{ef}$ are also $\F_q(\pi)$-linearly independent, which implies  $[K:\F_q(\pi)]\ge [K_v:\F_q((\pi))]$. 

This section is mainly devoted to proving Theorem \ref{Property A}. First we need some lemmas.  

\begin{lem}\label{Necessity of A} 
Suppose $\pi\in A$ and $\Gamma\subset A$. If $[K:\F_q(\pi)]>[K_v:\F_q((\pi))]$, then $(\Gamma,\pi)$ can not have property A. 
\end{lem} 

\begin{proof}
Write $h:=[K:\F_q(\pi)]$. Choose $\alpha_1,\ldots,\alpha_{ef}\subset A$ such that their reductions modulo $\pi A$ are linearly independent over $\F_q$. Then $\alpha_1,\ldots,\alpha_{ef}$ are linearly independent over $\F_q(\pi)$. Extend $\alpha_1,\ldots,\alpha_{ef}$ to an $\F_q(\pi)$-basis of $K$, say $\alpha_1,\ldots,\alpha_{ef},\alpha_{ef+1},\ldots,\alpha_{h}\in A$. Then $x\in K$ if and only if there exist uniquely determined elements $x_1,\ldots,x_h\in \F_q(\pi)$ such that $x=x_1\alpha_1+\cdots+x_h\alpha_h$.  
In particular, for $\gamma\in\Gamma\subset A$, write  
$$\gamma:=\gamma^{(1)}\alpha_1+\cdots+\gamma^{(h)}\alpha_h,$$ 
where $\gamma^{(i)}\in \F_q(\pi)$. 
Note that though $v(\gamma)\ge 0$, it is possible $v(\gamma^{(i)})<0$. For any $\gamma\in \Gamma$, define 
$$m(\gamma):=\frac{1}{e}\min_{1\le i\le h} \{v(\gamma^{(i)})\}.$$ 
Then each $\gamma^{(i)}$ is of the form 
$$\gamma^{(i)}=\sum_{n\ge m(\gamma)}b_{n,\gamma}^{(i)}\pi^n,$$ 
where $b_{n,\gamma}^{(i)}\in\F_q$, and $b_{m(\gamma),\gamma}^{(i)}\neq 0$ for at least one $i$.  Let 
$$m(\Gamma):=\min_{\gamma\in\Gamma} \{m(\gamma)\}.$$ 
Then obviously $m(\Gamma)\le 0$. Choose $\gamma'\in \Gamma$ such that $m(\gamma')=m(\Gamma)=:m$. Set  
$$\Lambda_{\gamma'}:=\{(\lambda_1,\ldots,\lambda_h)\in \F_q^h: \lambda_1\alpha_1+\cdots+\lambda_h\alpha_h\equiv \gamma'\mod \pi A\}.$$ 
Since the reductions of $\alpha_1,\dots,\alpha_{ef}$ consist of an $\F_q$-basis of $A/\pi A$, for any $\lambda_{ef+1},\dots, \lambda_h \in \F_q$, there exists a unique $(\lambda_1,\dots,\lambda_{ef})$ such that 
$(\lambda_1,\ldots,\lambda_h)\in\Lambda_{\gamma'}$. So $|\Lambda_{\gamma'}|=q^{h-ef}$. Therefore whether $(b_{m,\gamma'}^{(1)},\ldots,b_{m,\gamma'}^{(h)})$ belongs to $\Lambda_{\gamma'}$ or not, we may choose $(x_1,\ldots,x_h)\in\Lambda_{\gamma'}$ that is different from $(b_{m,\gamma'}^{(1)},\ldots,b_{m,\gamma'}^{(h)})$. Take 
\begin{equation} \label{linear expression i}
    x=x_1\alpha_1+\cdots+x_h\alpha_h.
\end{equation} 
Then $x\equiv \gamma\mod \pi A$. Let 
$$x=\sum_{n\ge 0}a_n\pi^n$$ 
be the expansion of $x$ with respect to $(\Gamma,\pi)$. Then $a_0=\gamma'$. Write each $a_n$ as 
$$a_n=a_n^{(1)}\alpha_1+\cdots+a_n^{(h)}\alpha_h,$$
where $a_n^{(i)}=\gamma^{(i)}$ for $1\le i\le h$ if $a_n=\gamma\in\Gamma$. Then  
\begin{equation}\label{linear expression ii}
x=\sum_{n\ge 0}\sum_{i=1}^h a_n^{(i)}\alpha_i \pi^n=\sum_{i=1}^h \alpha_i \sum_{n\ge 0}a_n^{(i)}\pi^n. 
\end{equation}

Suppose $(\Gamma,\pi)$ has property A. Then $\{a_n\}_{n\ge 0}$ is ultimately periodic. Since $a_n^{(1)},\ldots,a_n^{(h)}$ are uniquely determined by $a_n$, $\{a_n^{(i)}\}_{n\ge 0}$ is also ultimately periodic for $1\le i\le h$. This implies $(1-\pi^L)\sum_{n\ge 0}a_n^{(i)}\pi^n\in \F_q(\pi)$ for some positive integer $L$, and hence $\sum_{n\ge 0}a_n^{(i)}\pi^n\in \F_q(\pi)$. Hence both (\ref{linear expression i}) and (\ref{linear expression ii}) are linear expressions of $x$ with respect to the basis $\alpha_1,\ldots,\alpha_h$. Thus we have 
$$x_i=\sum_{n\ge 0}a_n^{(i)}\pi^n {\rm ~for~ } 1\le i\le h.$$ 
Recall $a_0=\gamma'$ and $m=m(\gamma')=m(\Gamma)\le 0$. We deduce  
$$\sum_{n\ge 0}a_n^{(i)}\pi^n-b_{m,\gamma'}^{(i)}\pi^m=x_i-b_{m,\gamma'}^{(i)}\pi^m\in \pi^{m+1}\F_q[[\pi]],$$
which could not happen if $m<0$. So we must have $m=0$. It then implies $m=0$ and $x_i=b_{m,\gamma'}^{(i)}$ as both $x_i$ and $b_{m,\gamma'}^{(i)}$ are in $\F_q$. This contradicts to $(x_1,\ldots,x_h)\neq (b_{m,\gamma'}^{(1)},\ldots,b_{m,\gamma'}^{(h)})$. Thus $(\Gamma,\pi)$ can not have property A. the proof of Lemma \ref{Necessity of A} is finished. 
\end{proof}

Let $\pi\in A$. For $\alpha_1,\ldots,\alpha_{ef}\in A$ whose reductions modulo $\pi A$ are linearly independent over $\F_q$, the set 
$$\F_q\alpha_1+\cdots+\F_q\alpha_s=\{c_1\alpha_1+\cdots+c_{ef}\alpha_{ef}: c_1,\ldots,c_{ef}\in \F_q\}$$ 
is a complete set of representatives of $A$ modulo $\pi A$. We show that if $[K:\F_q(\pi)]=[K_v:\F_q((\pi))]$, then $(\F_q\alpha_1+\cdots+\F_q\alpha_{ef},\pi)$ has property A.

\begin{lem}\label{sufficiency step i}
Suppose $\pi\in A$ and $[K:\F_q(\pi)]=[K_v:\F_q((\pi))]$. For $\alpha_1,\ldots,\alpha_{ef}\in A$ such that their reductions modulo $\pi A$ are linearly independent over $\F_q$, let $\Gamma=\F_q\alpha_1+\cdots+\F_q\alpha_s$. 
Then $(\Gamma,\pi)$ has property A. 
\end{lem}

\begin{proof}
Since $\alpha_1,\ldots,\alpha_{ef}\in A$ consist of an $\F_q((\pi))$-basis of $K_v$, there exist $x_1,\ldots,x_{ef}\in \F_q((\pi))$ such that $x=x_1\alpha_1+\cdots+x_{ef}\alpha_{ef}$. Write each $x_i$ as 
$$x_i:=\sum_{n\ge m}a_n^{(i)}\pi^n, a_n^{(i)}\in \F_q.$$ 
By letting $a_n:=a_n^{(1)}\alpha_1+\cdots+a_n^{(ef)}\alpha_{ef}\in \Gamma$, we derive  
$$x=\sum_{n\ge m}(a_n^{(1)}\alpha_1+\cdots+a_n^{(ef)}\alpha_{ef})\pi^n=\sum_{n\ge m}a_n\pi^n.$$ 
So $\sum_{n\ge m}a_n\pi^n$ is the expansion of $x$ with respect to $(\Gamma,\pi)$.  

Note that under the hypothesis $[K:\F_q(\pi)]=[K_v:\F_q((\pi))]$, $\alpha_1,\ldots,\alpha_{ef}$ also consist of an $\F_q(\pi)$-basis of $K$. We deduce that $x\in K$ if and only if $x_1,\ldots,x_{ef}\in \F_q(\pi)$, if and only if each sequence $\{a_n^{(i)}\}_{n\ge m}$ is ultimately periodic, if and only if $\{a_n\}_{n\ge m}$ is ultimately periodic. Hence $(\Gamma,\pi)$ has property A. This shows Lemma \ref{sufficiency step i}. 
\end{proof}

Assume $\pi\in A$, $\Gamma\subset A$, and $[K:\F_q(\pi)]=[K_v:\F_q((\pi))]$. Let $\alpha_1,\ldots,\alpha_{ef}$ be elements of $A$ whose reductions modulo $\pi A$ are linearly independent over $\F_q$. Then $\alpha_1,\ldots,\alpha_{ef}$ consist of both $\F_q((\pi))$-basis of $K_v$ and $\F_q(\pi)$-basis of $K$. Each $\gamma\in \Gamma$ can be write as 
$$\gamma=\gamma^{(1)}\alpha_1+\cdots+\gamma^{(ef)}\alpha_{ef}$$ 
with $\gamma^{(i)}\in \F_q(\pi)$. We show that $(\Gamma,\pi)$ has property $A$ if for every $\gamma\in\Gamma$, all its components $\gamma^{(i)}$ are polynomials in $\pi$. 

\begin{lem}\label{sufficiency step ii}
Let $\pi\in A$, $\Gamma\subset A$, and $[K:\F_q(\pi)]=[K_v:\F_q((\pi))]$. Let $\alpha_1,\ldots,\alpha_{ef}$ be elements of $A$ whose reductions modulo $\pi A$ are linearly independent over $\F_q$. If $\Gamma\subset\F_q[\pi]\alpha_1+\cdots+\F_q[\pi]\alpha_{ef}$, then $(\Gamma,\pi)$ has property A. 
\end{lem}

\begin{proof} 
For $x\in K_v$, let $x=\sum_{n\ge m}a_n\pi^n$ be the expansion of $x$ with respect to $(\Gamma,\pi)$. If $\{a_n\}_{n\ge m}$ is ultimately periodic, then it is easy to see $x$ is rational. 

Conversely let $x$ be rational. We need to show $\{a_n\}_{n\ge m}$ is ultimately periodic. For $\gamma\in \Gamma$, write $\gamma=\gamma^{(1)}\alpha_1+\cdots+\gamma^{(ef)}\alpha_{ef}$, where $\gamma^{(i)}\in \F_q(\pi)$ for $1\le i\le ef$. Since $\Gamma\subset \F_q[\pi]\alpha_1+\cdots+\F_q[\pi]\alpha_{ef}$,  
each $\gamma^{(i)}$ is a polynomial in $\pi$. Set  
$$d:=\max_{\gamma\in\Gamma, 1\le i\le ef} \deg \gamma^{(i)},$$ 
where $\deg \gamma^{(i)}$ is the degree of $\gamma^{(i)}$ as a polynomial in $\pi$ over $\F_q$. Then we can write each $a_n\in \Gamma$ as $a_n:=\sum_{i=1}^{ef} a_n^{(i)}\alpha_i$, where $a_n^{(i)}$ is of the form  
$$a_n^{(i)}=\sum_{k=0}^d a_{n,k}^{(i)}\pi^k,\ a_{n,k}^{(i)}\in \F_q.$$  
Thus 
$$x=\sum_{n\ge m}a_n\pi^n=\sum_{n\ge m} \pi^n \sum_{i=1}^{ef} a_n^{(i)}\alpha_i 
=\sum_{i=1}^{ef} \alpha_i \sum_{n\ge m} (\sum_{k=0}^d a_{n,k}^{(i)}\pi^k)\pi^n
=\sum_{i=1}^{ef} \alpha_i \sum_{n\ge m} \pi^n \sum_{k=0}^d a_{n-k,k}^{(i)},$$ 
where we set $a_{n-k,k}^{(i)}=0$ if $n-k<m$. Let $\Gamma':=\F_q\alpha_1+\cdots+\F_q\alpha_{ef}$. Denote by 
$$b_{n}^{(i)}:=\sum_{k=0}^d a_{n-k,k}^{(i)}\in\F_q, \quad 
{\rm and} \quad  
b_n:=\sum_{i=1}^{ef} b_n^{(i)}\alpha_i\in \Gamma'.$$ 
Then   
$$x=\sum_{i=1}^{ef} \alpha_i \sum_{n\ge m} b_n^{(i)}\pi^n=\sum_{n\ge m}b_n\pi^n$$ 
is the expansion of $x$ with respect to $(\Gamma',\pi)$. Therefore by Lemma \ref{sufficiency step i}, $\{b_n\}_{n\ge m}$ is ultimately periodic. Hence there are integers $N\ge m+d$ and $L>0$ such that 
\begin{equation}\label{eq2.1}
b_n=b_{n+L} \ {\rm for\ all }\ n\ge N.
\end{equation}

Fix any $n\ge N$. Consider $(d+1)$-tuples 
$$(a_{n-d+iL},a_{n-d+1+iL},\ldots,a_{n+iL})\in \Gamma^{d+1}$$ 
for $0\le i\le r^{d+1}$. Since $|\Gamma^{d+1}|=r^{d+1}$, there must exist $0\le i_1<i_2\le r^{d+1}$ such that 
$(a_{n-d+i_1L},\ldots,a_{n+i_1L})$ and $(a_{n-d+i_2L},\ldots,a_{n+i_2L})$ are equal. Take $n_0=n+i_1L$ and $j=i_2-i_1$. Then 
\begin{equation}\label{eq2.2}
(a_{n_0-d},\ldots,a_{n_0})=(a_{n_0-d+jL},\ldots, a_{n_0+jL}). 
\end{equation}

Now we show that (\ref{eq2.1}) and (\ref{eq2.2}) imply $a_{n_0+1}=a_{n_0+1+jL}$. By (\ref{eq2.1}), we have   
$$b_{n_0+1}=\sum_{k=0}^d a_{n_0+1-k,k}^{(i)}=\sum_{k=0}^d a_{n_0+1-k+jL,k}^{(i)}=b_{n_0+1+jL}. $$ 
Meanwhile by applying (\ref{eq2.2}) we may get  
$$a_{n_0+1-k,k}^{(i)}=a_{n_0+1-k+jL,k}^{(i)}$$ 
for $1\le k\le d$ and $1\le i\le {ef}$. It then follows that 
$$a_{n_0+1,0}^{(i)}=a_{n_0+1+jL,0}^{(i)}$$ 
for $1\le i\le {ef}$. Since there exists exactly one $\gamma\in\Gamma$ such that 
$$\gamma\equiv \sum_{i=1}^{ef} a_{n_0+1,0}^{(i)}\alpha_i \mod M,$$ 
we derive $a_{n_0+1}=\gamma=a_{n_0+1+jL}$. Finally an easy induction argument gives us $a_{n}=a_{n+jL}$ for all $n\ge n_0$. That is, $\{a_n\}_{n\ge m}$ is ultimately periodic. So we conclude that $(\Gamma,\pi)$ has property A. Lemma \ref{sufficiency step ii} is proved. 
\end{proof}

Now we are ready to give the proof of Theorem \ref{Property A}. 

\subsection*{Proof of Theorem \ref{Property A}} 
Assume $(\Gamma,\pi)$ has property A. For $\gamma\in\Gamma$, if 
$$\gamma=\sum_{n\ge m}a_n\pi^n$$
is the expansion of $\gamma$ with respect to $(\Gamma,\pi)$, then $m=0$, $a_0=\gamma$, and $a_n=0$ for all $n>0$. So its sequence of coefficients is ultimately periodic. Thus $\gamma$ is rational. By the same argument we can show that $\gamma\pi$ is rational for any $\gamma\in \Gamma$, which implies $\pi$ is rational. The equation  $[K:\F_q(\pi)]=[K_v:\F_q((\pi))]$ follows from Lemma \ref{Necessity of A}. This shows necessity. 
\\

Conversely suppose $\pi\in A$, $\Gamma\subset A$, and $[K:\F_q(\pi)]=[K_v:\F_q((\pi))]$. We show $(\Gamma,\pi)$ has property A. Let $x\in K_v$ and $x=\sum_{n\ge m}a_n\pi^n$ be the expansion of $x$ with respect to $(\Gamma,\pi)$. If the sequence $\{a_n\}_{n\ge m}$ is ultimately periodic, then $x$ is rational. It remains to prove $\{a_n\}_{n\ge m}$ is ultimately periodic if $x$ is rational. 

Let $\alpha_1,\ldots,\alpha_{ef}\in A$ such that their reductions modulo $\pi A$ are linearly independent over $\F_q$. Then $\alpha_1,\ldots,\alpha_{ef}$ consist of an $B:=\F_q(\pi)\cap\F_q[[\pi]]$-basis of $A$. For $\gamma\in \Gamma$, write 
$$\gamma=\gamma^{(1)}\alpha_1+\cdots+\gamma^{(ef)}\alpha_{ef}, \gamma^{(i)}\in B.$$ 
Then there exists a positive integer $L$, depending only on $\Gamma$, such that $(1-\pi^L)\gamma^{(i)}\in\F_q[\pi]$ for all $\gamma\in\Gamma$ and $1\le i\le {ef}$. Let 
$$\Gamma'=(1-\pi^L)\Gamma=\{(1-\pi^L)\gamma:\gamma\in\Gamma\}.$$
Since $(1-\pi^L)\gamma\equiv\gamma\mod \pi A$ and $\Gamma$ is a complete set of representatives modulo $\pi A$, $\Gamma'$ is also a complete set of representatives modulo $\pi A$. Note that $\Gamma'\subset \F_q[\pi]\alpha_1+\cdots+\F_q[\pi]\alpha_{ef}$. Then by Lemma \ref{sufficiency step ii}, $(\Gamma',\pi)$ has property A. 

Denote by $b_n:=(1-\pi^L)a_n$. Then $b_n\in(1-\pi^L)\Gamma=\Gamma'$. Since $x=\sum_{n\ge m}a_n\pi^n$ is the expansion of $x$ with respect to $(\Gamma,\pi)$, 
$$(1-\pi^L)x=\sum_{n\ge m}(1-\pi^L)a_n\pi^n=\sum_{n\ge m}b_n\pi^n$$ 
is the expansion of $(1-\pi^L)x$ with respect to $(\Gamma',\pi)$. Hence $\{b_n\}_{n\ge m}=\{(1-\pi^L)a_n\}_{n\ge m}$ is ultimately periodic, implying that $\{a_n\}_{n\ge m}$ is ultimately periodic. Now we conclude that $(\Gamma,\pi)$ has property A. The proof of Theorem \ref{Property A} is ended. \hfill$\Box$ \\

\begin{rem}
For given field $K$, valuation $v$ and positive integer $e$, there does not necessarily exists $\pi\in K$ such that $[K:\F_q(\pi)]=[K_v:\F_q((\pi))]$. An interesting question is to decide for which $K,v$ and $e$ that such a $\pi$ exists. 
\end{rem}

\section{Expansions of algebraic elements} 

In this section we consider expansions of algebraic elements. For a nonnegative integer $k$ and any  $\gamma\in\Gamma$, if $\gamma\pi^k=\sum_{n\ge m}a_n\pi^n$ is the expansion of $\gamma\pi^k$ with respect to $(\Gamma,\pi)$, then $m=k$, $a_k=\gamma$, and $a_n=0$ for $n>k$. So $\{a_n\}_{n\ge m}$ is ultimately periodic. Hence $(\Gamma,\pi)$ has property B yields $\pi$ and all $\gamma\in\Gamma$ are algebraic. We conclude that

\begin{lem}\label{necessity of B}
    If $(\Gamma,\pi)$ has property B, then $\pi$ and elements of $\Gamma$ are algebraic. 
\end{lem}

This confirms the only if part of Conjecture \ref{Property B}. Moreover, by using property of automatic sequences, one direction of the if part of Conjecture \ref{Property B} is easy to verify. 

\begin{lem}\label{atm to alg}
Assume $\pi$ and elements of $\Gamma$ are algebraic. For $x\in K_v$, if its expansion with respect to $(\Gamma,\pi)$ is $p$-automatic, then $x$ is algebraic. 
\end{lem}

\begin{proof}
Let $x=\sum_{n\ge m}a_n\pi^n$ be the expansion of $x$ with respect to $(\Gamma,\pi)$. Without loss of generality, we may assume $x\in A_v$, i.e., $m=0$. If $\{a_n\}_{n\ge 0}$ is $p$-automatic, then for any $\gamma\in\Gamma$, 
$$x_{\gamma}:=\sum_{n\ge 0, a_n=\gamma}x^n$$ 
is algebraic. Hence $x=\sum_{\gamma\in\Gamma}\gamma x_{\gamma}$ is algebraic. This shows Lemma \ref{atm to alg}. 
\end{proof}

Now by Lemma \ref{necessity of B} and \ref{atm to alg}, to verify Conjecture \ref{Property B} it remains to verify the following statement: {\it if $x$ is algebraic, then its expansion with respect to $(\Gamma,\pi)$ is $p$-automatic, provided that $\pi$ and elements of $\Gamma$ are algebraic.} 

We do not prove this remaining part in present paper. But we will verify it under the condition that $\Gamma$ is additively closed. Recall that $\Gamma$ is additively closed if $\mu+\nu\in\Gamma$ for any $\mu,\nu\in\Gamma$. To proceed, we need a equivalent description for additively closed set. 

\begin{lem}\label{add-cls set}
A set $S\subset K_v$ is additively closed if and only if $S$ is an $\F_p$-linear space. 
\end{lem}

\begin{proof}
If $S$ is an $\F_p$-linear space, then $S$ is obviously additively closed. On the other hand, $K_v$ is an $\F_p$-linear space. If $S$ is an additively closed subset of $K_v$, then it is easy to see $S$ is closed for addition and scalar multiplication, hence is a $\F_p$-linear subspace of $K_v$. 
\end{proof} 

Let $L/K$ be a finite extension of global function fields of characteristic $p$. Let $w$ be any valuation of $L$ which extends the valuation $v$ of $K$, that is $w|_K=v$. Let $L_w$ be the completion of $L$ under $w$. Denote by $L^a$ the algebraic closure of $L$ in $L_w$ and $K^a$ the algebraic closure of $K$ in $K_v$. Obviously $L^a$ is also the algebraic closure of $K$ in $L_w$. Hence $K^a$ is a subfield of $L^a$. In the process of proving the additively closed case of Conjecture \ref{Property B}, we come across the following result. 

\begin{thm}\label{ext deg of alg cls}
We have $[L^a:K^a]=[L_w:K_v]$. 
\end{thm}

Before delivering the proof of Theorem \ref{ext deg of alg cls}, we show how it is used to prove the additively closed case of Conjecture \ref{Property B}.

\begin{thm}\label{add-cls case of B}.
    Let $\Gamma$ be additively closed. Then $(\Gamma,\pi)$ has property B if and only if $\pi$ and elements of $\Gamma$ are algebraic. 
\end{thm}

\begin{proof}
By the discussion right below Lemma \ref{atm to alg}, we need only to show that each algebraic element $x$ has $p$-automatic expansion with respect to $(\Gamma,\pi)$, provided that $\pi$ and elements of $\Gamma$ are algebraic. We do this in the following. 

Since $\Gamma$ is additively closed, by Lemma \ref{add-cls set}, $\Gamma$ is an $\F_p$-linear space. Let $u:=\log_p|\Gamma|$, and $\alpha_1,\ldots,\alpha_u$ consist of an $\F_p$-basis of $\Gamma$. Then the reductions of  $\alpha_1,\ldots,\alpha_u$ modulo $\pi$ consist of an $\F_p$-basis of $A/\pi A$. It follows that $\alpha_1,\ldots,\alpha_u$ are linearly independent over $\F_p((\pi))$, and hence are linearly independent over $\F_p(\pi)^a$, where $\F_p(\pi)^a$ denotes the algebraic closure of $\F_p(\pi)$ in $\F_p((\pi))$. Applying Theorem \ref{ext deg of alg cls} to the extension $K/\F_p(\pi)$, we deduce that 
$$u=\log_p q^{ef}=[K_v:\F_q((\pi))][\F_q((\pi)):\F_p((\pi))]=[K_v:\F_p((\pi))]=[K^a:\F_p(\pi)^a].$$ 
It follows that $\alpha_1,\ldots,\alpha_u$ consist of an $\F_p(\pi)^a$-basis of $K^a$. Hence for $x\in K^a$ there exist uniquely determined $x_1,\ldots,x_u\in \F_p(\pi)^a$ such that 
$$x=x_1\alpha_1+\cdots+x_u\alpha_u.$$
For $1\le i\le u$, write each $x_i$ as 
$$x_i=\sum_{n\ge m}a_n^{(i)}\pi^n, a_n^{(i)}\in \F_p\ {\rm for\ all\ } n\ge m.$$ 
By Christol's theorem, each sequence $\{a_n^{(i)}\}_{n\ge m}$ is $p$-automatic. Letting $a_n=\sum_{i=1}^u a_n^{(i)}\alpha_i$, we have  
$$x=\sum_{i=1}^u \alpha_i \sum_{n\ge m}a_n^{(i)}x^n=\sum_{n\ge m}\pi^n\sum_{i=1}^u a_n^{(i)}\alpha_i
=\sum_{n\ge m}a_n\pi^n.$$ 
It is easy to see $a_n\in\Gamma$ for all $\Gamma$, therefore $\sum_{n\ge m}a_n\pi^n$ is the expansion of $x$ with respect to $(\Gamma,\pi)$. Since $\{a_n^{(i)}\}_{n\ge m}$ are $p$-automatic for $1\le i\le u$, we deduce that $\{a_n\}_{n\ge m}$ is $p$-automatic. This finishes the proof of Theorem \ref{add-cls case of B}.
\end{proof} 

Now let us get back to the proof of Theorem \ref{ext deg of alg cls}. We show in advance two lemmas which are special cases of Theorem \ref{ext deg of alg cls}. 
\begin{lem}\label{separable case}
If $L/K$ is finite separable. Then $[L^a:K^a]=[L_w:K_v]$.
\end{lem}

\begin{proof}
Since $L/K$ is finite separate, there exists a primitive element $\alpha$ such that $L=K(\alpha)$. Furthermore, we see $L_w=K_v(\alpha)$ since both fields are complete, and $L\subset K_v(\alpha)\subset L_w$. As $\alpha$ is a separable element, we have $L_w/K_v$ is separable. For any $\beta\in L^a$, obviously the minimal polynomial of $\beta$ over $K_v$ has all its coefficients from $K^a$, hence is also the minimal polynomial of $\beta$ over $K^a$. So $\beta$ is separable over $K^a$ as it is separable over $K_v$. As a consequence, $L^a/K^a$ is separable. In addition the coincidence of the minimal polynomials of $\alpha$ over both $K_v$ and $K^a$ gives us $[L_w:K_v]=[K^a(\alpha):K^a]$. Hence to show $[L^a:K^a]=[L_w:K_v]$ is equivalent to show $[L^a:K^a]=[K^a(\alpha):K^a]$. Assume to the contrary that $[L^a:K^a]$ is strictly larger than $[K^a(\alpha):K]$. Then by $L^a/K^a$ is separable, there exists some $\beta\in L^a$ such that its minimal polynomial $g(x)$ over $K^a$ is of degree greater than $[K^a(\alpha):K^a]$. But $g(x)$ is also the minimal polynomial of $\beta$ over $K_v$, and hence is of degree no more than $[L_w:K_v]=[K^a(\alpha):K^a]$, a contradiction. Therefore we have $[L^a:K^a]=[L_w:K_v]$. 
\end{proof}

\begin{lem}\label{inseparable case}
    If $L/K$ is purely inseparable of degree $p$, then $[L^a:K^a]=[L_w:K_v]=p$. 
\end{lem}

\begin{proof}
If $L/K$ is purely inseparable of degree $p$, then $L^p=K$. Taking completions on both side gives us $L_w^p=K_v$. Let $F_s$ be the algebraic closure of $\F_p$ in $L_w$. Then $\F_s$ is isomorphic to the residue field of $L_w$. Taking any uniformizer $\lambda$ of $L_w$, we have $L_w=\F_s((\lambda))$ (cf. \cite[Theorem 5.10]{Wi}). Then $K_v=\F_s((\lambda^p))$. Now by Christol's theorem, $L^a$ is the collection of Laurent series $\sum_{n\ge m}a_n \lambda^n$ with $p$-automatic sequences of coefficients, while $K^a$ is the collection of Laurent series  $\sum_{n\ge m}a_n \lambda^{pn}$ with $p$-automatic sequences of coefficients. Write 
$$\sum_{n\ge m}a_n \lambda^n=\sum_{0\le i\leq p-1}\lambda^i\sum_{n\ge m, n\equiv i\mod p}a_n\lambda^{pn}.$$ 
By \cite[Theorem 6.8.1]{AS}, the subsequences $\{a_n: n\ge m, n\equiv i\mod p\}$ are automatic, hence by Christol' theorem $\sum_{n\ge m, n\equiv i\mod p}a_n\lambda^{pn}$ are elements of $K^a$ for $0\le i\le p-1$. So we drive that $1,\lambda,\dots,\lambda^{p-1}$ is an $K^a$-basis of $L^a$, which induces $[L^a:K^a]=[L_w:K_v]=p$. 
\end{proof}

Now we prove Theorem \ref{ext deg of alg cls}. 

{\it Proof of Theorem \ref{ext deg of alg cls}.}
Let $L_0$ be the separable closure of $K$ in $L$. Then $L/L_0$ is purely inseparable. There exists a chain of proper extensions 
$$K\subset L_0\subset \cdots \subset L_n=L,$$ 
where $L_i/L_{i-1}$ are purely inseparable of degree $p$ for $1\le i\le n$. Denote by $\widehat{L_i}$ be the completion of $L_i$, and $L_i^a$ the algebraic closure of $L_i$ in $\widehat{L_i}$. Then we have chains of extensions  
$$K_v\subset \widehat{L_0}\subset \cdots \subset \widehat{L_n}=L_w$$ 
and
$$K^a\subset L_0^a\subset \cdots \subset L_n^a=L^a.$$
Applying Lemma \ref{separable case} and \ref{inseparable case} to the above two chains gives us the desired equation. \hfill$\Box$

\subsection*{Examples} The following is a concrete case for Theorem \ref{add-cls case of B}. Let $P(z)\in\F_q[z]$ be any irreducible polynomial of degree $d>0$. For all $Q(z)\in\F_q[z]$, there exists an unique nonnegative integer $k$ such that $P^k(z)$ divides $R(z)$ but $P^{k+1}(z)$ does not divide $R(z)$. We define $v_P(Q):=k$ and extend $v_P$ to the rational function field $\F_q(z)$ by 
$$v_P(Q/R):=v_p(Q)-v_P(R).$$ 
Then $v_P$ is a valuation on $\F_q(z)$. We call $v_P$ the $P$-adic valuation of $\F_q(z)$. The completion of $\F_q(\pi)$ under $v_P$ is called the $P$-adic completion of $\F_q(z)$. 
Take $\pi=P(x)$ and let $\Gamma$ be the collection of polynomials over $\F_q$ of degree less than $d$, i.e., 
$$\Gamma:=\{c_0+c_1z+\cdots+c_{d-1}z^{d-1}:c_0,\ldots,c_{d-1}\in\F_q \}.$$  
Then $\Gamma$ is a complete set of representatives of $A/\pi A$. Additionally $\Gamma$ is additively closed. For an element $x$ of the $P$-adic completion of $\F_q(z)$, let  
$$x=\sum_{n\ge m}a_n P^n$$ 
be the expansion of $x$ with respect to $(\Gamma,P(z))$. Then $x$ is algebraic over $\F_q(z)$ if and only if $\{a_n\}_{n\ge m}$ is $p$-automatic. Specially, take $z=P(z)$, then $\Gamma=\F_q$, we get the classical Christol's theorem. \\

By the following example we illustrate that Conjecture \ref{Property B} may be true when $\Gamma$ is not additively closed. Suppose $\pi$ and elements of $\Gamma$ are algebraic. Let $\xi$ be an algebraic element in $A_v$. If $\Gamma$ is additively closed, then $\Gamma+\xi=\{\gamma+\xi: \gamma\in \Gamma\}$ is a complete set of representatives of $A_v/\pi A_v$ that is not additively closed. For any $x\in K_v$, let $x=\sum_{n\ge m}b_n\pi^n$ be its expansion with respect to $(\Gamma+\xi, \pi)$. We write $a_n=b_n-\xi\in\Gamma$ for each $\ge m$. Then 
$$x=\sum_{n\ge m}(a_n+\xi)\pi^n=\sum_{n\ge m} a_n\pi^n+\frac{\pi^m\xi}{1-\pi}.$$ 
Hence $\sum_{n\ge m} a_n\pi^n$ is the expansion of $x-\pi^m\xi/(1-\pi)$ with respect to $(\Gamma,\pi)$. So $x$ is algebraic if and only if $x-\pi^m\xi/(1-\pi)$ is algebraic, if and only if $\{a_n\}_{n\ge m}$ is $p$-automatic by Theorem \ref{add-cls case of B}, if and only if $\{b_n\}_{n\ge m}$ is $p$-automatic. 

\subsection*{An equivalent statement of Conjecture \ref{Property B}} In the end of this section, we exhibit an  equivalent but more concrete statement of Conjecture \ref{Property B}. Let $\F_s$ be the algebraic closure of $\F_q$ in $K_v$. Then $\F_s$ is isomorphic to the residue field of $K_v$. For any uniformizer $z$ of $K_v$, we have $K_v=\F_s((z))$ and $A_v=\F_s[[z]]$. So to find a $\pi\in A_v$ is equivalent to find an power series $P(z)$ in $\F_s[[z]]$, and to find a complete set of representatives $\Gamma$ of $A_v/\pi A_v$ from $A_v$ is equivalent to find a complete set of representatives of $\F_s[[z]]/P(z)\F_s[[z]]$ from $\F_r[[z]]$. Thus we have the following equivalent statement of Conjecture \ref{Property B}. \\
\\
{\bf Conjecture 1.2$'$} {\it Let $P(z)\in \F_s[[z]]$, and $\Gamma\subset \F_s[[z]]$. Suppose $\Gamma$ is a complete set of representatives of $\F_s[[z]]/P(z)\F_s[[z]]$. Then $(\Gamma, P)$ has property B if and only if $P(z)$ and elements of $\Gamma$ are algebraic over $\F_s(z)$.} \\
\\
As evidence to the conjecture we prove that 
\begin{thm}\label{case q=2}
Let $\Gamma=\{f_1,f_2\}\subset \F_2[[z]]$ be a complete residue system of $\F_2[[z]]$ modulo $z$. Let $x=\sum_{n\ge m}a_nz^n$ be the expansion of $x$ with respect to $(\Gamma,z)$. Then $(\Gamma,z)$ has property B if and only if $f_1,f_2$ are algebraic. 
\end{thm}

\begin{proof} 
As it is similar to the proof of Theorem \ref{add-cls case of B}, we need only show that if $f_1$ and $f_2$ are algebraic, then the algebraicity of $x$ implies the 2-automaticity of $\{a_n\}_{n\ge m}$.  

Assume $f_1,f_2$ are algebraic. Suppose the constant terms of $f_1$ and $f_2$ are $1$ and $0$ respectively. Denote by $a_n^{(0)}$ the constant term of $a_n$. Since $a_n\in\{f_1,f_2\}$, we have $a_n=f_1$ if $a_n^{(0)}=1$, and $a_n=f_2$ if $a_n^{(0)}=0$. Then we deduce that  
$$a_n=a_n^{(0)}f_1+(1-a_n^{(0)})f_2.$$
Thus 
$$x=\sum_{n\ge m}a_nz^n=\sum_{n\ge m}(a_n^{(0)}f_1+(1-a_n^{(0)})f_2)z^n
=(f_1-f_2)\sum_{n\ge m}a_n^{(0)}z^n+f_2\sum_{n\ge m}z^n.$$
Note that $x,f_1,f_2$ are algebraic, it follows that $\sum_{n\ge m}a_n^{(0)}z^n$ is also algebraic. By Christol's Theorem, the sequence $\{a_n^{(0)}\}$ is $2$-automatic. Therefore $\{a_n\}_{n\ge 0}$ is $2$-automatic. This finishes the proof of Theorem \ref{case q=2}. 
\end{proof}

\section{$\beta$-expansions of formal Laurent series}
In this section, we apply Theorem \ref{add-cls case of B} to study $\beta$-expansions of formal Laurent series. Let $\F_q((1/z))$ be the field of formal Laurent series of the form 
$$x=\sum_{n\ge -m}\frac{a_n}{z^n}, a_n\in\F_q,$$ 
where $m\in\mathbb{Z}$ is called the degree of $x$. We say a formal Laurent series is algebraic if it is algebraic over $\F_q(1/z)$. The set of Laurent series of degree no more than zero is $\F_q[[1/z]]$. The integer part $[x]$ and the fraction part $\{x\}$ of $x$ are respectively defined as 
$$[x]:=\sum_{-m\le n\le 0}\frac{a_n}{z^n}~~ {\rm and~~ } \{x\}:=\sum_{n>0}\frac{a_n}{z^n}.$$ 
Let $\beta$ be a series in $\F_q((1/z))$ of positive degree. Define 
$$T(x):=\beta x-[\beta x], \forall x\in\F_q[[1/z]].$$ 
Denote by $T^n$ the $n$-th iteration of $T$ for positive integers $n$, and set $T^0(x)=x$ for all $x\in\F_q[[1/z]]$. Then we have the following expansion 
$$x=\frac{a_1}{\beta}+\frac{a_2}{\beta^2}+\cdots,$$ 
where $a_n=[\beta T^nx]$ are elements of $\F_q[z]$. The $\beta$-expansion of $x$, denoted by $d_{\beta}(x)$, is defined to be the sequence $\{a_n\}_{n\ge 1}$. It is an analogue of $\beta$-expansions of real numbers introduced by R\'enyi \cite{Re}. Many properties about $\beta$-expansions are studied in \cite{HM,Sch,SS}, we concentrate on those related to automatic sequences. Hbaib and Mkaouar \cite[Theorem 5.4]{HM} proved that if $d_{\beta}(1)$ is automatic, then $\beta$ is algebraic. Assuming $\deg \beta=1$, they \cite[Theorem 5.6]{HM} further proved that $d_{\beta}(1)$ is $p$-automatic if and only if $\beta$ is algebraic. Scheicher and Sirvent \cite[Theorem 6.1]{SS} showed that if $\beta$ is isolated of positive degree, then $d_{\beta}(1)$ is $p$-automatic. They \cite[Theorem 6.6]{SS} also showed that if $\beta$ is a Pisot or Salem series, then $d_{\beta}(\alpha)$ is $p$-automatic if and only if $\alpha$ is algebraic. Based on above results, the following questions on $\beta$-expansions appear in \cite{HM} and \cite{SS}. 

Question 1 \cite[Remark 6.7]{SS}: Is either the converse of \cite[Theorem 5.4]{HM} or the converse of  \cite[Theorem 6.1]{SS} true. More precisely, whether one of the following statements is true: (i) $d_{\beta}(1)$ is $p$-automatic if and only if $\beta$ is algebraic; and (ii) $d_{\beta}(1)$ is $p$-automatic if and only if $\beta$ is isolated. 

Question 2 \cite{HM}: For any algebraic $\beta$ of positive degree, is it true that $d_{\beta}(x)$ is $p$-automatic if and only if $x$ is algebraic. 

Question 3 \cite[Remark 6.7]{SS}: Is the converse of \cite[Theorem 6.6]{SS} true. That is, for any $\beta$ of positive degree, if $d_{\beta}(x)$ is $p$-automatic if and only if $x$ is algebraic, then is $\beta$ a Pisot or Salem series.

We answer all the three questions by proving that 

\begin{thm}\label{beta expansions}
    Let $\beta\in\F_q((1/z))$ be a formal Laurent series of positive degree. Then the following statements are equivalent: 
    
    ${\rm (i)}$ $\beta$ is algebraic;

    ${\rm (ii)}$ $d_{\beta}(1)$ is $p$-automatic; 
    
    ${\rm (iii)}$ for any $x\in \F_q[[1/z]]$, $d_{\beta}(x)$ is $p$-automatic if and only if $x$ is algebraic.
\end{thm}

\begin{proof}
Since $\F_q(1/z)$ is a subfield of $\F_q((1/z))$, $\deg(y)$ is well defined for any $y\in\F_q(1/z)$. It is easy to check that $-\deg:\F_q(1/z)\rightarrow \mathbb{Z}$ is a valuation on $\F_q(1/z)$. Set $K=\F_q(1/z)$ and $v=-\deg$. Then $K_v=\F_q((1/z))$ and $A_v=\F_q[[1/z]]$. Take $\pi=1/\beta\in \F_q[[1/z]]$, $\Gamma=\{c_0+c_1/z+\cdots c_{d-1}/z^{d-1}: c_i\in\F_q\}$. Then $\Gamma$ is an additively closed complete set of representatives of $A_v/\pi A_v$. 

Write $d=\deg \beta$. Let $\{a_n\}_{n\ge 1}$ be the $\beta$-expansion of $x\in \F_q[[1/z]]$. Then by \cite[Lemme 2.8]{HM} or \cite[Theorem 3.3]{SS}, $\deg{a_n}<d$. So $a_n/z^d\in\Gamma$ for all $n\ge 1$. Hence  
$$\frac{\beta x}{z^d}=\sum_{n=1}^{\infty} \frac{a_n/z^d}{\beta^{n-1}}
=\sum_{n\ge 1}\frac{a_n}{z^d}\cdot \pi^{n-1}$$
is the expansion of $\beta x/z^d$ with respect to $(\Gamma,\pi)$. Then $\beta$ is algebraic $\Rightarrow$ $\{a_n\}_{n\ge }$ is $p$-automatic if and only if $\{a_n/z^d\}_{n\ge 1}$ is $p$-automatic, if and only if $\beta x/z^d$ is algebraic by Theorem \ref{add-cls case of B} $\Rightarrow$ $d_{\beta}(1)$ is $p$-automatic $\Rightarrow$ $\beta$ is algebraic by \cite[Theorem 5.4]{HM}. That is (i)$\Rightarrow$ (iii) $\Rightarrow$ (ii) $\Rightarrow$ (i). This ends the proof of Theorem 4.1. 
\end{proof}

{\bf Acknowledgement:} This paper is supported by National Natural Science Foundation of China (No.12271382). The author would like to thank Lian Duan, Yining Hu and Shaoshi Chen for many helpful discussions.

\bibliographystyle{amsalpha}

\end{document}